\documentclass[10pt,openright]{article}
\usepackage{amssymb, amsmath}
\usepackage{color}
\usepackage{latexsym}
\usepackage{amscd}
\usepackage{amsmath,amsfonts,amssymb,amsbsy,amstext,amsthm,amscd}
\usepackage{graphicx}

\newtheorem{theo}{Theorem}
\newtheorem{coro}{Corollary}
\newtheorem{propo}{Proposition}

\begin{document}

\title{ Moufang Theorem for a  variety of local  non-Moufang  loops}

\author{Ramiro Carrillo-Catal\'an \\
\small CONACyT - Universidad Pedag\'ogica Nacional \\[-0.8ex]
\small Unidad 201 Oaxaca\\[-0.8ex]
\small \texttt{rcarrilloca@conacyt.mx}\\
Marina Rasskazova\\
\small Omsk State Institute of Service\\[-0.8ex]
\small Pevtsova street 13, Omsk\\[-0.8ex]
\small \texttt{marinarasskazova@yandex.ru}\\
Liudmila Sabinina\\
\small Centro de Investigacion en Ciencias   \\[-0.8ex]
\small UAEM, Cuernavaca\\[-0.8ex]
\small \texttt{liudmila@uaem.mx}
\date{May 26, 2017}
}

\maketitle

\begin{abstract}
An open problem in theory of loops is to find the variety of non- Moufang loops satisfying the Moufang Theorem. 
In this note, we present  a variety of local smooth diassociative loops with such  property.  \\
 

\end{abstract}

\section{Introdution}

At  the conference ``Loops '11"  in Trest, Czech Republic,  Andrew Rajah proposed the following question: \\
\textbf{(Moufang theorem in non-Moufang loops).} We say that a variety $V$ of loops satisfies the Moufang theorem
 if for every loop $Q$ in $V$ the following implication holds: 
 if $x(yz) = (xy)z$ for every $x, y, z \in Q$ then the subloop generated by $x, y, z$ is a group. 
Is every variety satisfying the Moufang theorem contained in the variety of Moufang loops?\\

Recently this question was discussed in  several articles, as for example \cite {ColbSt}, \cite{MGColb},  and \cite{St}.

In all cited  articles there were found some cases or clases of Steiner loops satisfying the  statement  of Moufang Theorem.
Non of those examples present the variety of loops.  Our example of a variety in the context of  local smooth loops 
 doesn't give the complete answer to  the question of A. Rajah, but it sheds some additional light on his question.

 R. Moufang,   in 1931, started studying algebraic structures which today are called Moufang loops.
Recall that a \textit{loop}  is a set $Q$ with a binary operation  $<Q, \cdot, 1>$,such that  the equations $x\cdot a= b,\quad a\cdot y=b$
have a unique solution  $ \forall  a,b \in Q$ and $x\cdot 1=x=1\cdot x \quad  \forall x\in Q.$ And in the same way, a \textit{Moufang loop} is a loop in which the   identity 
  $ ((xy)x)z=x(y(xz))$ hold for every three elements of the loop.   
A loop $Q$ is called  {\textit diassociative } if every two elements generate a subgroup of $Q$.
 In addition to  Moufang loops 
another example of diassociative loops are {\textit Steiner} loops. Steiner loops are diassociative commutative loops of exponent $2$. (See [ColbSt]  and [MGColb]).

\section {Main Theorem}
  In 1935  Ruth Moufang  showed   her famous theorem [RM]:
Let $M$ be a Moufang loop. If $(a, b, c) = ((ab)c) (a(bc))^{-1}=1$ for some elements $ a,b,c \in M$ then $a, b, c$ generate a subgroup of $M$. 
It is easy to see that  the Moufang Theorem implies the diassociativity of loops: to verify this fact, consider the identity $(a, b, 1)=1$ in the loop, 
which  obviously always has a place.\\

In 1955  A.I. Malcev applying the  Campbell-Hausdorff formula to the varieties of smooth  local loops
introduced   the Binary-Lie algebras  as  tangent algebras of  smooth diassociative loops and Moufang-Lie algebras 
( now called Malcev algebras )  as tangent algebras of smooth local Moufang loops  \cite{Ma}. 
The identities $ x^2 = 0,\, J(x,y, xy)=0,$ (where  
 $J(x,y,z)=(xy)z+(yz)x+(zx)y$) defining the variety of Binary-Lie algebras were found in \cite{Ga}. 
On the other hand, the identities $ x^2 = 0,\, J(x,y, xz)=J(x,y, z)x,$  which define the variety of Malcev algebras were 
stated in \cite {Sa}.   The  identities  $ x^2 = 0,\, J(x,y, zt)=J(x,y, z)t=0$ of the variety  of  Malcev algebras
which are tangent algebras of smooth local left-automorphic  Moufang loops  are discussed in \cite {CS1}  and 
the identities  $ x^2 = 0, J(x,y, xz)=J(x,y, z)x=0$  of the variety  of Malcev algebras which are tangent algebras of  
smooth local almost left automorphic Moufang loops were found in  \cite {CS2}\\

The analog of  Moufang Theorem in the context of Malcev algebras has the following form:
Let $\mathcal M$  be a Malcev algebra. If  for given three  elements  $\{x_1, x_2, x_3\}$ of $\mathcal M$  the equality
$J(x_1,x_2,x_3)=0$ is satisfied, the subalgebra  of $\mathcal M$ generated by these elements, $\{x_1, x_2, x_3\}$,  is a Lie algebra.\\

Thus, Rajah's question in this sense should be:  Is there a variety of Binary-Lie algebras satisfying the  analog of Moufang theorem  which does  not belong to  the variety of Malcev algebras ?  \\

In the following, all the algebras will be consider as algebras over a field $\it k$.

The aim of this note is to show the following:
\begin{theo}
Let ${\mathfrak w}$ be a  variety of  Binary-Lie  algebras defined by the  identities
$$  x^2 = 0,\quad J( x, y, zu) = 0,$$
 where  
 $J(x,y,z)=(xy)z+(yz)x+(zx)y$. \\
Then\\
1. ${\mathfrak w} $ is not a variety of Malcev algebras.\\
2.  Any algebra of the variety   ${\mathfrak w}$  satisfies the statement of  an analog of  Moufang Theorem. 
\end{theo}
\begin{proof}
1.
Consider  a non-nilpotent  solvable  $4$-dimensional algebra $\mathcal L$  from the variety   ${\mathfrak w}$ generated by the elements
$\{a, b, c\}$ with the following relations:\\
$$ab=ac=0,\quad bc=d,\quad da=d,\quad bd=cd=0$$
we have $ J(a, b, c)=d$, and therefore  by direct computation  we get:
$$J(\mathcal L) = {\mathcal L}^2 = kd = Lie(\mathcal L),$$
where $k\in \it k,$
$J(\mathcal L)$ is an ideal generated by all {\it jacobians} $J(x,y,z)  \forall x,y,z \in \mathcal{L}$ and 
$Lie(\mathcal L)= \{ x\in \mathcal L\mid J(x, y, z) = 0, \forall y, z \in \mathcal L\}$ is a {\it Lie center} of an algebra $\mathcal L.$
\\
We have $J(a, b, ac ) - J(a, b, c)a =d \neq 0, $ hence 
 $L$ is not a Malcev algebra.\\[4ex]

2. In order to prove the second statement of the Theorem
let us consider the  algebra $\mathcal C$ from the variety ${\mathfrak w}$ , generated by the elements $\{a,b,c\},$ such that 
 the following condition  holds:
$$J(a, b, c)=0.$$
Let us note that in this case $$J(w_1, w_2, w_3)=0,$$
for all $w_i \in  \mathcal C$, 
Indeed, we have
$ w_i =v_i  + u_i$ where $v_i$ is an element from the vector space $V $ with the base $\{a, b, c\} $
and $u_i \in {\mathcal C}^2$. Therefore  by definition of ${\mathfrak w}$  we have 
$$J(w_1, w_2, w_3) = J(v_1, v_2, v_3) = \alpha J(a, b, c)=0,\quad \alpha \in \it k.$$
\end{proof}

Using the Malcev Theorem  on the correspondence between local diassociative loops  and their tangent Binary-Lie algebras \cite {Ma} we get the following

\begin{coro}
There exists the  non-Moufang variety of local dissociative loops ${\mathfrak W}$ such that the statement of Moufang Theorem holds 
for every  loop in ${\mathfrak W}.$
\end{coro}

Now we give  the general construction for all  algebras from  the variety
${\mathfrak w}$.

Let $L$ be a Lie algebra, $Der \,\, L$ the Lie algebra of derivations of $L$ 
and  $P$ an arbitrary vector space.
Let $P_0 \subseteq Der L$ be a subspace of $Der L$ such that $[P_0,P_0] \subseteq Inn L,$
where $ Inn L$ is the ideal of all {\it inner derivations} of $L$, i.e. $Inn L =\{ad x \mid x\in L\} \quad (ad x: a \rightarrow [ax]).$ \\
Let $\psi :P\rightarrow P_0$ be some epimorphism.\\
Consider  $L=L_0 \oplus Z(L),$ where $Z(L)= \{ x\in L\mid [x,L]=0\}.$ It is possible to identify:
$$ L_0 \cong L/ Z(L)\cong Inn L.$$
Suppose that $[p_1,p_2]\in L_0, \,\,p_i\in P.$\\
In addition, consider an arbitrary linear function $\lambda :P\wedge P \rightarrow Z(L).$ \\

On the set $ B= P\oplus L$ we define the operation:
\begin{align}
 (p_1,l_1)\cdot (p_2,l_2) = (0, [p^{\psi}_1,p^{\psi}_2] + \lambda (p_1,p_2) + l^{p^{\psi}_2}_1 - l^{p^{\psi}_1}_2 +[l_1,l_2]).
\end{align}

\begin{propo} The algebra $B$ with the operation (1) belongs to the variety  ${\mathfrak w}.$ Any algebra from the 
 variety ${\mathfrak w}$ can be obtained by the  construction described above.

\end{propo} 
\begin{proof}
1. An algebra $B$ belongs to the variety ${\mathfrak w}$ if and only if 
 $B^2 $ belongs to the Lie center $Z(B).$\\
 Since  $B\cdot B \subseteq L$.  it is enough to show that $ L\subseteq Lie(B).$
By  construction \\
$J(p_1, a, b)=0$, since $p_1$ acts on $L$ as a derivation. In the same way, $J(p_1, p_2, a) = 0,$
 since $\lambda (p_1,p_2)\in Z(L)$ for all $p_1, p_2 \in P,\,\,\,a, b\in L.$\\
[4ex]

2. Now we show that every algebra from the variety ${\mathfrak w}$ may be obtained using the general construction.\\[2ex]

Consider  $B\in {\mathfrak w}$  and denote $L = Lie(B)$ , then $B= L\oplus P$, where $P$ is some vector space. By definition of ${\mathfrak w}\quad B^2 \subseteq L$, in particular $ [P, P] \subseteq L.$ For any $p\in P$ denote 
$$ p^{\psi}\in Der L:  p^{\psi}(a) = [p,a].$$

Notice that $\psi$ is defined correctly, since  $L$ is a Lie center. Moreover  $P_0 = P^{\psi}\subseteq Der L$ is satisfies the condition 
$$[P_0,P_0] \subseteq Inn L \cong L/Z(L)$$
In the case  $L= Z(L)\oplus L_0,$ where $L_0$ is a suitable vector space, we have a map 
$$ \phi : P_0\wedge P_0 \rightarrow L_0, \quad \phi (p^{\psi}_1 \wedge  p^{\psi}_2) = l_0,$$
where $[p_1,p_2] = l_0 +z. \quad l_0\in L_0, z\in Z(L)$,\\
Finally, define  $\lambda :P\wedge P \rightarrow Z(L)\quad \lambda (p_1 \wedge p_2) = z.$  
Under our notations
$[Z(L),p] \subseteq Z(L), \, p\in P$ implies the correctness of the definition of the map $\phi$.
This way one can construct the algebra $B$ using the Lie algebra $L$ and two maps defined above:
$\psi: P\rightarrow P_0$ and $ \lambda : P\wedge P \rightarrow Z(L).$ 
\end{proof}
{\bf Example.}\\
Let $L=\kappa c, \kappa \in \it k$  be a one-dimensional Lie algebra, let $P$ be a vector space, generated  by  elements $\{t, a, b\}.$\
Consider $P_0 = Der L.$
$$P_0=\kappa t_0,\, \kappa \in \it k,\quad ct_0 = c, \quad  \psi(t)=t_0,\quad \psi (a) = \psi( b) = 0,$$
$$\lambda (t, a) = \alpha_1c,\quad \lambda (t, b) = \alpha_2c,\quad \lambda (a,b) = \alpha_3c$$
Let us denote  $ B(\alpha_1, \alpha_2, \alpha_3)$ the  corresponding Binary-Lie algebra. 
It is easy to see that if $\alpha_3 \neq 0$, then every  such an algebra is isomorphic to $B(0, 0, 1)$
with base $\{t, a, b, c\}$ and operation:
$$ at=bt=0 \quad ab=c,\quad ac=bc=0,\quad ct=c.$$\\

Consider the variety ${\mathfrak v}$  of algebras defined by the identities $$x^2 =0, \quad J(x,y,xz)=0.$$ The variety 
${\mathfrak v}$  is the variety of Binary-Lie algebras, which are not  Malcev algebras, because  
${\mathfrak w}$  is contained in ${\mathfrak v}$. 
Now,  lets consider a free algebra of the variety ${\mathfrak v}$  generated by the elements $\{a,b,ac\}.$
We know that  $J(a,b, (ac)) = 0$.  Then, one can conjecture that $J(a,b, (ab)(ac)) \neq 0$.
If this conjecture is true then it will shown that the variety ${\mathfrak v}$  does not obey the statement of Moufang Theorem.\\

 Some natural questions arise:\\

1. Is it possible to find a maximal subvariety of the variety of Binary-Lie algebras for which the analog of Moufang Theorem holds?\\

2.  Analysing our example of  the  variety  ${\mathfrak W }$ of local loops  one can  consider   the variety   $  \Omega$   of  
(discrete) diassociative loops  with the additional identity  
$$(x,y,[z,t])=1.$$
It is known that the variety of Steiner loops   is  a subvariety of $\Omega.$ It is also known  that not all Steiner loops obey the statement of Moufang Theorem. 
In the case of  the variety of Steiner loops as we mentioned above the identity  $x^2 =1$   holds.

In this context, we conjecture that we can find a positive answer of A.Rajah's question for the variety of dissociative loops with the additional identities
$$(x,y,[z,t])=1, \quad x^m =1$$
for some odd $ m.$

\section*{Acknowledgments}
The authors thank  Alexander Grishkov and Ivan Shestakov for useful discussions.  

\section*{Funding}

The first author thanks  CONACYT and Universidad Pedag\'ogica Nacional Unidad 201 Oaxaca for supporting the C\'atedras CONACYT project  1522.   
The second author thanks  
to CNPq (Brazil), processo 307824/2016-0.
The third author thanks SNI and FAPESP grant process 2015/07245-4 for support.

\vspace{1cm}

\vspace{1.5cm}

\end{document}